\documentclass[11pt]{amsart}
\usepackage{amssymb,latexsym,amsfonts}
\usepackage{amsmath}
\usepackage{amscd}
\usepackage{pb-diagram,pb-xy}

\usepackage{graphicx}
\usepackage{hyperref}
\usepackage{color}
\usepackage[matrix,arrow]{xy}

\usepackage{fullpage}
\newcommand{\Diam}{{\rm Diam }}

\newcommand{\trace}{{\rm trace}\,}

\newcommand{\Ein}{{\mathrm {Ein}}}
\newcommand{\cEin}{{\mathbf {Ein}}}
\newcommand{\ein}{{\mathrm {ein}}}
\newcommand{\cein}{{{\mathbf {ein}}}}
\newcommand{\Eink}{{\mathrm {Ein}_k}}
\newcommand{\Schoutk}{{\mathrm {Sch}_k}}

\newcommand{\Einl}{{\mathrm {Ein}_l}}

\newcommand{\Eint}{{\mathrm {Ein}_2}}

\newcommand{\vol}{{\rm {Vol}}}

\newcommand{\totscal}{{\rm {TotalScal}}}

\newcommand{\Scal}{{\rm Scal}}

\newcommand{\Ric}{{\rm Ric}}

\newcommand{\str}{{\rm string}}
\DeclareMathAlphabet{\mathpzc}{OT1}{pzc}{m}{it}

\def\Z{{\mathbf Z}}

\newtheorem{theorem}{Theorem}[section]
\newtheorem{corollary}[theorem]{Corollary}

\newtheorem{proposition}[theorem]{Proposition}
\newtheorem*{thm-a}{Theorem\! A}
\newtheorem*{cor-a}{Corollary\! A}
\newtheorem*{thm-b}{Theorem\! B}
\newtheorem*{cor-b}{Corollary\! B}
\newtheorem*{thm-c}{Theorem\! C}
\newtheorem*{cor-c}{Corollary\! C}
\newtheorem*{thm-d}{Theorem\! D}
\newtheorem*{cor-d}{Corollary\! D}
\newtheorem*{thm-e}{Theorem\! E}
\newtheorem*{cor-e}{Corollary\! E}
\newtheorem*{thm-f}{Theorem\! F}
\newtheorem*{cor-f}{Corollary\! F}

\newtheorem*{conj-d}{Conjecture D}
\newtheorem*{conj-c}{Conjecture C}
\newtheorem{thm}{Theorem}[section]
\theoremstyle{definition}
\newtheorem{definition}[thm]{Definition}

\newtheorem*{remark}{Remark}

\newtheorem*{examples}{Examples}

\begin{document}
\author{Mohammed Labbi} 
\address{Department of Mathematics\\
 College of Science\\
  University of Bahrain\\
  32038, Bahrain.}
\email{mlabbi@uob.edu.bh}

\title{On  modified Einstein tensors and two smooth invariants of compact manifolds}  \maketitle
 \begin{abstract}
Let $(M,g)$ be a Riemannian $n$-manifold, we denote by $\Ric$ and $\Scal$ the Ricci and the scalar curvatures of $g$. For scalars $k<n$, the modified Einstein tensors denoted $\Eink$  are defined as  $\Eink :=\Scal \, g -k\Ric$.
Note that the usual Einstein tensor coincides with the half of $\Eint$ and ${\rm Ein}_0=\Scal.g$. It turns out that all these new modified tensors, for
$0<k<n$, are still  gradients of the total scalar curvature functional but with respect to  modified integral scalar products.
 The positivity of $\Eink$ for some positive $k$ implies the positivity of all ${\rm Ein}_l$ with $0\leq l\leq k$ and so we define a smooth invariant $\cEin(M)$ of $M$  to be the supremum of positive k's that renders $\Eink$ positive. By definition $\cEin(M)\in [0,n]$, it is zero if and only if $M$ has no positive scalar curvature metrics and it is maximal equal to $n$ if $M$ possesses an Einstein metric with positive scalar curvature. In some sense, $\cEin(M)$ measures how far is $M$ to admit an Einstein metric of positive scalar curvature. In this paper we prove that $\cEin(M)\geq 2$ if $M$ admits an effective action by a non abelian connected Lie group or if $M$ is simply connected of positive scalar curvature and dimension $\geq 5$. We prove as well that the invariant $\cEin$ increases after a surgery operation on the manifold $M$ or by assuming that the manifold $M$ has higher connectivity. We prove that the condition $\cEin(M)\leq n-2$ does not imply any restriction on the first fundamental group of $M$. We define and prove similar properties for an analogous invariant namely $\cein(M)$. The paper contains several open questions.
\end{abstract}
\tableofcontents
\section{Introduction}

Let $(M,g)$ be a Riemannian manifold of dimension $n$. We denote as usual by $\Ric$ and $\Scal$ the Ricci curvature and the scalar curvature. The Einstein tensor $\Ein:=\frac{1}{2}\Scal\, g -\Ric$ is the gradient of the total scalar curvature functional. In this paper, we slightly modify the  Einstein tensor and we define the {\sl modified Einstein tensors}
\begin{equation}
\Eink :=\Scal \, g -k\Ric,
\end{equation}
Where $k$ is a constant. We recover the (double of) Einstein tensor for $k=2$. All these new modified tensors , for
$0<k<n$, are still  gradients of the same total scalar curvature functional but with respect to  modified integral scalar products, see proposition \ref{Bourg} below.
\subsection{Positivity properties of modified Einstein tensors}
We are interested in the positivity properties of these tensors and to their effect on the topology of the manifold. Note that
\begin{equation}
\trace(\Eink)=(n-k)\Scal.
\end{equation}
Consequently, for $k<n$, the positivity of $\Eink$ implies the positivity of the scalar curvature. Note that for $k>n$, $\Eink>0 \Rightarrow \Scal <0$, furthermore,  for $k=n$, $\Eink$ is, up to a factor, the trace free Ricci  tensor. It cannot have a constant sign unless the metric $g$ is Einstein with zero scalar curvature. For these reasons, we will restrict our study to the cases $k<n$. The following proposition shows that the positivity property is hereditary and its proof is straightforward
\begin{proposition}\label{hered}
\begin{enumerate}
\item For any real numbers $0<k<l<n$ we have $$\Einl>0 \Rightarrow \Eink>0\Rightarrow \Scal>0.$$
\item For any real numbers $k<l<0$ we have $$\Eink>0 \Rightarrow \Einl>0\Rightarrow \Scal>0.$$
\end{enumerate}

\end{proposition}
\begin{examples}
\begin{enumerate}
\item For a Riemannian $n$-manifold  with  constant sectional curvature $\lambda$, $\Eink=(n-1)(n-k)\lambda g$. In particular, it is positive if $n>1$,  $k<n$ and $\lambda>0$.
\item For an Einstein $n$-manifold  with  scalar curvature $\rho$, $\Eink=\rho (n-k)\lambda g$. In particular, it is positive if $k<n$ and $\rho>0$.
\end{enumerate}
\end{examples}
\subsection{ The $\Ein$ and $\ein$ invariants of  Riemannian metrics}
In view of proposition \ref{hered}, it is natural to define the following scalars
\begin{definition}
For a Riemannian metric $g$ on a fixed $n$-manifold, we define the scalars
\begin{equation}
\Ein(g):=\sup\{k \in (0, n): \Eink>0\}\,\,\, {\mathrm and}\,\,\, \ein(g):=\inf\{k<0:\Eink>0\}.
\end{equation}
We set $\Ein(g)=\ein(g)=0$  if the scalar curvature of $g$ is not positive and $\ein(g)=-\infty$ in case the corresponding set of $k$'s is unbounded below. 
\end{definition}
Recall that by $\Eink>0$, we mean that it is positive definite at each point of the manifold. Of course one may alternatively, define a scalar curvature functions, by taking the above supremum and infinimum pointwise  that is at each point of  the manifold.\\
For a compact Riemannian manifold $(M,g)$ with positive scalar curvature one can check without difficulties that
\begin{equation}
\Ein(g)=\inf_{x\in M}\frac{\Scal(x)}{\rho_{\mathrm {max}}(x)},
\end{equation}
where $\rho_{\mathrm {max}}(x)$ denotes the highest eigenvalue of Ricci at $x\in M$. In particular, $\Ein(g)>0$ for any metric with positive scalar curvature on a compact manifold. Note that a similar formula holds for $\ein(g)$ if one replaces $\inf$ by $\sup$ and $\rho_{\mathrm {max}}$ by $\rho_{\mathrm {min}}$.\\
As an immediate consequence of the previous formula we have
\begin{proposition}
For a compact Riemannian $n$-manifold $(M,g)$ we have
\[ \Ein(g)=n\iff g \,\, {\rm {is\,\, an\,\, Einstein\,\, metric\, \,with\,\, positive\,\, scalar\,\, curvature.}}\]
\end{proposition} 
We will see in the sequel of the paper a similar extremal property for $\ein(g)$, namely $\ein(g)=-\infty$ if and only if $g$ has nonnegative Ricci curvature and positive scalar curvature. \\

In the following example of Berger spheres, we will  show that it is possible for $\Ein(g)$ to be any value between $0$ and $n$ and for $\ein(g)$ to take any value from $0$ to $-\infty$.\\
Let $\pi:S^{2n+1}\rightarrow P^n {\Bbb C}$ be the Hopf fibration of the
sphere of dimension $2n+1$ by circles $S^1$ over the complex
projective space of real dimension $2n$. Let the sphere be endowed
with the standard Riemannian metric $g$ with curvature 1 and the
complex projective space be endowed with the Fubiny-Study metric
$\check{g}$ which renders the projection $\pi$ a Riemannian
submersion. We then shrink the fibers (circles) by multipling the
metric by $t>0$ in the directions tangent to the fibers. The result is
a new Riemannian metric $g_t$, called the canonical variation of $g$,
on $S^{2n+1}$ and the projection $\pi:S^{2n+1}\rightarrow P^n {\Bbb
  C}$ remains a Riemannian submersion. It turns out that for each $t>0$, the Ricci
curvature of the metric $g_t$  has one
eigenvalue $\rho_1=2nt$ of multiplicity one and another eigenvalue
$\rho_2=-2t+2n+2$ of multiplicity $2n$, see for instance \cite{Besse}. In particular, the Ricci curvature of $g_t$ is positive if and only if $0<t<n+1$. The scalar curvature of $g_t$ is
constant and is equal to $-2nt+2n(2n+2)$.  Consequently, the $\Eink$ tensors of the metric $g_t$ have one eigenvalue $\nu_1=4n(n+1)-2nt(1+k)$ of multiplicity one and a second eigenvalue $\nu_2=2t(k-n)+(2n-k)(2n+2)$ of multiplicity $2n$. One can then immediately check the following
\begin{proposition} The berger metrics $g_t$ on the sphere $S^{2n+1}$ satisfy the following properties
\begin{itemize}
\item For $t\geq 2n+2$, $\Ein(g_t)=\ein(g_t)=0$.\\
\item For $1\leq t<2n+2$, $\Ein(g_t)=\frac{2n+2}{t}-1$, it ranges between $2n+1$ and $0$.\\
\item For $0<t<1$, $\Ein(g_t)=n\bigl(1+\frac{2n+2}{-2t+2n+2}\bigr)$, it ranges between $2n$ and $n+1$.\\
\item For $n+1< t<2n+2$, $\ein(g_t)=n\bigl(1+\frac{2n+2}{-2t+2n+2}\bigr)$, it ranges between $0$ and $-\infty$.
\item For $0<t\leq n+1$, $\ein(g_t)=-\infty.$
\end{itemize}
\end{proposition}

\begin{remark} The previous example shows in particular that the positivity of Ricci curvature does not imply in general the positivity of $\Eink$ for $k>1$. Precisely,  for  $\frac{2(n+1)}{k+1}<t<n+1$, the Berger metric $g_t$ has positive Ricci curvature but all the $\Eink$ curvatures are not positive for $1<k<2n+1$.
\end{remark}

\subsection{ The $\cEin$ and $\cein$ invariants of  compact manifolds}
For a given compact $n$-manifold, denote by ${\mathcal M}$ the space of all  Riemannian metrics on $M$. We  define two smooth invariants $\cEin(M)$ and $\cein(M)$  of $M$ as follows
\begin{equation}
\begin{split}
\cEin(M)=& \sup\{\Ein(g):  g\in {\mathcal M}\},\\
\cein(M)=& \inf \{\ein(g):  g\in {\mathcal M}\}.
\end{split}
\end{equation}
Note that by definition $\cEin(M)=\cein(M)=0$ if $M$ does not admit any metric with positive scalar curvature. If $M$ admits a positive scalar curvature metric then $\cEin(M)>0$ and $\cein(M)<0$, if furthermore the metric has nonnegative Ricci curvature then $\cein(M)=-\infty$. In the case $M$ possesses an Einstein metric with positive scalar curvature then $\cEin(M)=n$. In particular, it is important to remark that
\begin{center}
A compact $n$-manifold $M$ with $\cEin(M)<n$ does not admit Einstein metrics of positive type.
\end{center}
 In general, $\cEin(M)\in [0,n]$ and $\cein(M)\in [-\infty, 0].$\\
In this paper we are interested in the study of compact manifolds with the property $\cEin(M)>k$ for some $k\in [0, n)$. For $k=0$, that condition is equivalent to positive scalar curvature and was extensively studied by Gromov, Lawson, Schoen, Yau, Stolz and many others. We remark that
\[ \cEin(M)>k \iff \exists g\in {\mathcal M}\,\, {\rm with}\,\, \Ein_k(g)>0.\]
\subsection{Statement of the results}
The property $\Eink >0$  behaves very well with Riemannian products and as well with Riemannian submersions as shown by the next theorem.\\
Let $(M,g)$ be a total space of a Riemannian submersion $p: M\to
B$. We denote by $\hat g$ the restriction of $g$ to
the fiber $F=p^{-1}(x)$, $x\in B$. \\
For a Riemannian submersion $p: M\to B$ as above, there is a
\emph{canonical variation} $g_t$ of the original metric $g$, which is
a fiberwise scaling by $t^2$. Then, if the fiber metrics $\hat g$
satisfies some bounds, this construction delivers a metric on a total
space with positive (negative) curvature.

\begin{thm-a} 
Let $M$ be a compact manifold that admits a Riemannian metric which make it the 
 total space of a Riemannian submersion $\pi: M\to
B$. Denote by $p$ the dimension of the fibers.
We suppose that all the fibers with the induced metric have positive $\Eink$ curvature for some $k$.
\begin{enumerate}
\item If $0<k<p$ then $\cEin(M)>k$.
\item If $k<0$ then $\cein(M)<k$.
\end{enumerate}
In particular, the cartesian product of two compact manifolds $M_1$ and $M_2$ satisfies
\[\cEin(M_1\times M_2)\geq \max\{\cEin(M_1),\cEin(M_2)\}\,\, {\rm and}\,\, \cein(M_1\times M_2)\leq \min\{\cein(M_1),\cein(M_2)\}\]
\end{thm-a}
For instance, If $M$ is any compact manifold and $S^2$ is the two sphere then
\[\cEin(S^2\times M)\geq 2\,\, {\rm and}\,\, \cein(S^2\times M)=-\infty.\]

Another application of the above theorem is the following generalization of Lawson-Yau theorem \cite{Lawson-Yau} about positive scalar curvature
\begin{thm-b}
If a compact connected manifold $M$ admits  an effective action of a non-abelian  compact connected Lie group  $G$ then $\cEin(M)\geq 2$ and $\cein(M)=-\infty$.
\end{thm-b}

\begin{thm-c}
\begin{enumerate}
\item[a)] Suppose a compact manifold $N$ is obtained from a manifold $M$ by surgeries of codimensions $\geq 3$ then
\[\cein(N)\leq \cein(M).\]
Furthermore, if $\cEin(M)\leq 2$ then $\cEin(N)\geq \cEin(M)$.
\item[b)] Suppose $\cEin(M)> 2$ and the manifold $N$ is obtained from $M$ by surgeries of codimensions $\geq \cEin(M)+1$, then
\[\cEin(N)\geq \cEin(M).\]
\end{enumerate}
\end{thm-c}
A direct consequence of the above theorem is the following
\begin{cor-c}
\begin{enumerate}
\item[a)] The connected sum of two compact manifolds $M_1$ and $M_2$ of dimensions $\geq 3$ satisfies
\[\cein(M_1\#M_2)\leq \max\{\cein(M_1),\cein(M_2)\}.\]
\item[b)] The connected sum of two compact manifolds $M_1$ and $M_2$ of dimensions $\geq 3$ and such that $\cEin(M_i)\leq 2$ for $i=1,2$ satisfies
\[\cEin(M_1\#M_2)\geq \min\{\cEin(M_1),\cEin(M_2)\}.\]
\item[c)] The connected sum of two compact $n$-manifolds $M_1$ and $M_2$ such that $2<\cEin(M_i)\leq n-1$ for $i=1,2$  satisfies
\[\cEin(M_1\#M_2)\geq \min\{\cEin(M_1),\cEin(M_2)\}.\]

\end{enumerate}
\end{cor-c}

\begin{thm-d}\label{simply-connected-thm}
Let $n\geq 5$ and $k\in(-\infty, 2)$. Then
\begin{enumerate}
\item Any compact simply connected and non-spin $n$-manifold $M$ with $n>4$ satisfies the properties
\[  \cEin(M) \geq 2 \,\, {\rm and}\,\,  \cein(M)=-\infty.\]

\item For a compact simply connected spin $n$-manifold $M$ with $n>4$ we have the equivalence of the following properties
\[ \cEin(M)>0\iff \cEin(M) \geq 2\iff \cein(M)=-\infty.\]
\end{enumerate}
\end{thm-d}

As we increase the connectivity of the manifold psc  can be upgraded for free up to 3 or 4 as follows
\begin{thm-e}
\begin{enumerate}
\item For a compact $2$-connected  $n$-manifold $M$ with $n>6$ we have the following equivalence 
\[ \cEin(M)>0\iff \cEin(M) \geq 3.\]
\item For a compact $3$-connected and non-string  $n$-manifold $M$ with $n>8$ we have the following equivalence 
\[ \cEin(M)>0\iff \cEin(M) \geq 4.\]
\end{enumerate}
\end{thm-e}

\begin{thm-f} Let $\pi$ be any finitely presented group and $n>3$.
\begin{enumerate}
\item For every $k\in(0, n-2)$, there exists a compact $ n$-manifold $M$ with $\cEin(M)>k$ and $\pi_1(M)=\pi$.
\item  For every $k\in(-\infty, 0)$, there exists a compact $ n$-manifold $M$ with $\cein(M)<k$ and $\pi_1(M)=\pi$.
\end{enumerate}
\end{thm-f}
\subsection{Plan of the paper}
The next section is section 2 where we will study some properties of the modified Einstein tensors. In section 2.1 we prove that each modified Einstein tensor $\Eink$  is the gradient of the total scalar curvature functional once restricted to metrics with unit volume with respect to a modified integral scalar product. In section 2.2 we emphasize on the microscopic effect of the positivity of some $\Eink$ on the volume of small geodesic spheres and the volume and total scalar curvature of small tubes around curves in the manifold. Section 2.3 is about relations between the positivity of $\Eink$ and the positivity of other curvatures. In particular, we prove the following implications
\begin{itemize}
\item For $1\leq k\leq n-1$, the positivity of $\Eink$ implies that the Ricci curvature is $(n-k)$-positive. In particular,  $\Eint>0$ implies $(n-2)$-positive Ricci curvature and ${\rm Ein}_{n-1}>0$ implies the positivity of the Ricci curvature.
\item  For $k$ integer with $1\leq k\leq n-1$, the positivity of $\Eink$ implies that the Ricci operator has at least $k+1$ positive eigenvalues.
\item For $n\geq 4$ and $k=\frac{2n(n-1)}{3n-4}$, the positivity of $\Eink$ implies the positivity of the $\Gamma_2(A)$ curvature, where $A$ is the Schouten tensor. That is to say that $\sigma_1(A)>0$ and $\sigma_2(A)>0$.
\item  For $n\geq 4$ and $k=\frac{2(n-1)^2}{2n-3}$, the positivity of $\Eink$ implies the positivity of the Schouten tensor $A$. It implies in particular the positivity of all $\Gamma_i(A)$ curvatures.
\end{itemize}
In section 3, we provide proofs of all theorems stated in the introductory section 1. In section 4, we study and determine in several cases the $\cEin$ and $\cein$ invariants of compact manifolds of dimension $\leq 4$. The last section 5 is about some remarks and several open questions.

\section{Properties of modified Einstein tensors}
\subsection{Modified Einstein tensors as gradients of the total scalar curvature functional}
Recall that the usual Einstein tensor is the gradient of the total scalar curvature functional. Precisely, let $F(g)=\int_M\Scal(g)\mu_g$ be defined on the space of Riemannian metrics, then the directional derivative at $g$ in the direction of a symmetric tensor field $h$ is given by
\[ F'_gh=\langle \frac{1}{2}\Scal\, g-\Ric,\, h\rangle \]
where $\langle .,.\rangle=\int_Mg(.,.)\mu_g$ is the canonical integral scalar product on symmetric tensor fields induced by the metric $g$. Note that we denoted by $g$ the Riemannian metric on the manifold and also its extension to symmetric tensors.\\
Next we modify the above integral scalar product and define for $\alpha\in \Bbb{R}$ and for two symmetric tensor fields $h_1$ and $h_2$ the following
\[\langle h_1,h_2\rangle_\alpha=\int_M\bigl\{ g(h_1,h_2)-\alpha \,g(g,h_1)\, g(g,h_2)\bigr\}\mu_g=\langle h_1,h_2\rangle -\alpha \int_M({\rm trace}_gh_1)({\rm trace}_gh_2)\mu_g.\]
It follows directly from the Cauchy-Schwartz inequality that for $\alpha <\frac{1}{n}$, the scalar product $\langle h_1,h_2\rangle_\alpha$ is positive definite. In fact, for any symmetric tensor field $h$ on $M$ one has
\[\langle h,h\rangle_\alpha=\int_M\bigl\{ g(h,h)-\alpha \,g(g,h)\, g(g,h)\bigr\}\mu_g\geq (1-\alpha n) \int_M g(h,h)\mu_g.\]
\begin{proposition}\label{Bourg}
For $n\geq 3$ and for each $k$ with $0< k<n$, the modified Einstein tensor $\Eink$ on a compact $n$-manifold  is the gradient of the total scalar curvature functional  with respect to the modified integral scalar product $\langle .,.\rangle_\alpha$. Precisely, we have
\[ F'_gh=\frac{1}{k}\langle \Eink,\, h\rangle_\alpha, \]
where $\alpha=\frac{k-2}{2(k-n)}\in (-\infty, \frac{1}{n})$.
\end{proposition}
\begin{proof}
Using the above notations we have
\begin{equation*}
\begin{split}
\langle \Scal\, g,\, h\rangle_\alpha=&\langle \Scal\, g,\, h\rangle-\alpha\int_Mg(\Scal\, g,g)\, g(g,h)\mu_g=\langle \Scal\, g,\, h\rangle-\alpha\int_M n\, \Scal \,g(g,h)\mu_g\\
=& \langle \Scal\, g,\, h\rangle-n\alpha\langle \Scal\, g,h\rangle=(1-n\alpha) \langle \Scal\, g,h\rangle.
\end{split}
\end{equation*}
Consequently, one has the following
\begin{equation*}
\begin{split}
F'_gh=&\langle \frac{1}{2}\Scal\, g-\Ric,\, h\rangle=\langle \frac{1}{2}\Scal\, g-\Ric,\, h\rangle_\alpha+\alpha\int_Mg\bigl(\frac{1}{2}\Scal\, g-\Ric,g\bigr)g(g,h)\mu_g\\
=&\langle \frac{1}{2}\Scal\, g-\Ric,\, h\rangle_\alpha+\alpha\int_M \frac{n-2}{2}\Scal\, g(g,h)\mu_g=\langle \frac{1}{2}\Scal\, g-\Ric,\, h\rangle_\alpha+
\alpha \frac{n-2}{2}\langle \Scal\, g,h\rangle\\
=& \langle \frac{1}{2}\Scal\, g-\Ric,\, h\rangle_\alpha+ \alpha \frac{n-2}{2(1-n\alpha)}\langle \Scal\, g,h\rangle_\alpha=\langle \frac{1-2\alpha}{2(1-n\alpha)}\Scal\, g-\Ric,\, h\rangle_\alpha.
\end{split}
\end{equation*}
To complete the proof take $k=\frac{2(1-n\alpha)}{1-2\alpha}$ so that  $\alpha=\frac{k-2}{2(k-n)}$. The later is a strictly decreasing function in $k$ and it ranges between $-\infty$ and $\frac{1}{n}$ as $k$ ranges between $0$ and $n$. This completes the proof.
\end{proof}
\begin{remark}
The previous proposition was first noticed by Bourguignon, see remark 3.17 in page 52 of \cite{Bourguignon}.
\end{remark}

\subsection{The sign of modified Einstein tensors and volume comparison}

Recall that a Riemannian $n$-manifold $(M,g)$ has positive scalar curvature (equivalently $\Ein(g)>0$) if and only if   the volume of all sufficiently small  geodesic $(n-1)$-spheres $S^{n-1}(r)$ in $M$,  are smaller than the volume of the equidimensional Euclidean $(n-1)$-dimensional spheres of the same radii in the Euclidean space ${\Bbb{R}}^n$.\\
A similar result holds for the $\Eint$ curvature as follows \cite{Labbi-einstein,agag}
\begin{proposition}
The Einstein tensor $\Eint(g)$ of a Riemannian $n$-manifold $(M,g)$ is positive (equivalently $\Ein(g)>2$)  if and only if  the volume of all sufficiently small  geodesic  $(n-2)$-spheres $S^{n-2}(r)$ in $M$ are smaller than the volume of the equidimensional Euclidean $(n-2)$-dimensional spheres of the same radii in the Euclidean space ${\Bbb{R}}^n$.
\end{proposition}
We precise that by a   geodesic  $(n-2)$-sphere $S^{n-2}(r)$ in $M$ we mean the image under the exponential map at a point of the intersection of a tangent  $(n-1)$-sphere and a tangent hyperplane, see \cite{Labbi-einstein}.\\
To interpret the sign of other modified Einstein tensors we need to consider tubes around curves instead of geodesic spheres. Let $\sigma$ be a curve in  Euclidean space ${\Bbb R}^n$ with finite length $L(\sigma)$. It turns out that the volume of a tube of radius $r$ around the curve $\sigma$ in ${\Bbb R}^n$ depends only on the radius $r$ and the length $L(\sigma)$. Precisely one has
\[\vol\left(T_{\sigma}^{{\Bbb R}^n}(r)\right)=\omega_{n-2}(r)L(\sigma).\]
Where $\omega_{k}(r)$ denotes the volume of the $k$-dimensional round sphere of radius $r$ in ${\Bbb R}^{k+1}$. We note that the previous formula is a special case of Weyl's tube formula \cite{Tubes} about the volume of tubes around  embedded submanifolds in Euclidean space.\\
Similar results hold for the total scalar curvature. Denote by $\totscal\left(T_{\sigma}^{{\Bbb R}^n}(r)\right)$ the total scalar curvature a tube of radius $r$ around a curve $\sigma$ in ${\Bbb R}^n$, then it is shown in \cite{totscal} that it does depend only on the radius of the tube and the length of the curve as follows
\[\totscal\left(T_{\sigma}^{{\Bbb R}^n}(r)\right)=(n-2)(n-3)r^{n-4}\omega_{n-2}(1)L(\sigma).\]
Where $n\geq 4$ and $\omega_{k}(1)$ denotes the volume of the $k$-dimensional unit sphere in ${\Bbb R}^{k+1}$.\\
The following proposition shows the microscopic effect of the positivity of ${\rm Ein}_{-1}$
\begin{proposition}
If the modified Einstein tensor ${\rm Ein}_{-1}(g)$ of $(M,g)$ is positive (or equivalently $\ein(g)<1$) then for any curve $\sigma$ in $M$ with finite length and for any sufficiently small $r$ we have
\[\vol\left(T_{\sigma}^{M}(r)\right)<\vol\left(T_{\sigma}^{{\Bbb R}^n}(r)\right).\]
Where  $\vol\left(T_{\sigma}^{M}(r)\right)$ denotes the volume of a tube of radius $r$ around the curve $\sigma$ in the manifold $M$. 
\end{proposition}
\begin{proof}
This is a direct consequence of Hotelling's formula \cite{Hotelling,Tubes} which can be written in the following form
\[\vol\left(T_{\sigma}^{M}(r)\right)=\omega_{n-2}(r)\left( L(\sigma)-\frac{r^2}{6(n+1)}\int_\sigma {\rm Ein}_{-1}(\sigma'(t),\sigma'(t))dt +O(r^4)\right).\]
Where we assumed that the curve $\sigma$ has unit speed.
\end{proof}
The positivity effect at microscopic level of other ${\rm Ein}_{-k}$ is illustrated in the following
\begin{proposition}
If the modified Einstein tensor ${\rm Ein}_{-\frac{n+2}{n-4}}(g)$ of a Riemannian manifold $(M,g)$ of dimension $n>4$  is positive (equivalently $\ein(g)<-\frac{n+2}{n-4}$) then for any curve $\sigma$ in $M$ with finite length and  for any sufficiently small $r$ we have
\[\totscal\left(T_{\sigma}^{M}(r)\right)<\totscal\left(T_{\sigma}^{{\Bbb R}^n}(r)\right).\]
Where  $\totscal\left(T_{\sigma}^{M}(r)\right)$ denotes the total scalar curvature of a tube of radius $r$ around the curve $\sigma$ in the manifold $M$. 
\end{proposition}
\begin{proof}
This is a direct consequence of the formula of Theorem 5.2 in \cite{totscal} which asserts that
\[\totscal\left(T_{\sigma}^{M}(r)\right)=\omega_{n-2}(1)r^{n-4}\left( (n-3)(n-2)L(\sigma)-\frac{r^2(n-3)}{6(n-1)(n-4)}\int_\sigma {\rm Ein}_{-\frac{n+2}{n-4}}(\sigma'(t),\sigma'(t))dt +O(r^4)\right).\]
Where we assumed that the curve $\sigma$ has unit speed.
\end{proof}

\subsection{Positive Ricci curvature and positive $\Ein_k$ curvatures}

\begin{proposition}\label{k-positive-Ricci}
Let $k\in [1, n-1]$ be a given integer. A Riemannian $n$-manifold has $(n-k)$-positive Ricci curvature if and only if it has $k$-positive $\Eink$ curvature.\\
In particular, positive Ricci curvature is equivalent to $(n-1)$-positive ${\rm Ein}_{n-1}$ curvature. Also $(n-2)$-positive Ricci curvature is equivalent to 2-positive ${\rm Ein}_2$ curvature, that is 2-positive Einstein tensor.
\end{proposition}
Recall that an operator on a real vector space is said to be $k$-positive if the sum of its lowest $k$ eigenvalues is positive.
\begin{proof}
Denote by $\rho_i$ the eigenvalues of Ricci curvature, then the sum of $k$ eigenvalues of $\Eink$ curvature is given by
\begin{equation}
k\Scal-k\sum_{a=1}^k\rho_{i_{a}}=k\sum_{a=k+1}^n\rho_{i_{a}}.
\end{equation}
This completes the proof.

\end{proof}
\begin{corollary}
For a Riemannian $n$-manifold $(M,g)$ one has
\begin{itemize}
\item $\Ein(g)>n-1\implies \Ric >0$.
\item $\Ein(g)>n-2\implies \Ric$ is $2$-positive.
\end{itemize}
\end{corollary}
We remark here that the standard metric $g$ on the product $S^1\times S^{n-1}$ has $\Ein(g)=n-1$ but the product $S^1\times S^{n-1}$ does not admit any metric with positive Ricci curvature.\\

\begin{proposition} 
A Riemannian metric has nonnegative Ricci curvature and positive scalar curvature if and only if it has positive  $\Eink$ curvatures  positive for all $k<0$, that is if and only if $\ein(g)=-\infty$.
\end{proposition}
\begin{proof}
The direct implication is straightforward as follows
\begin{equation}
\Eink=\Scal. g -k\Ric \geq \Scal.g>0.
\end{equation}
Conversely, Suppose that $\Scal.g-k\Ric>0$ for all $k<0$, then $\Scal >0$ and $\Ric>\frac{\Scal}{k}.g$ for all negative $k$. It follows that $\Ric \geq 0$.

\end{proof}

\begin{proposition}
For each $0<k<n-1$ we have 
\begin{equation}
\Eink>0 \implies (n-1)-positive \,\, \Eink \iff {\rm Ein}_{\frac{-k}{n-1-k}}>0.
\end{equation}

\end{proposition}

\begin{proof}
Remark that $\Eink>0 $ is $(n-1)$-positive if and only if the first Newton transformation $t_1(\Eink)=\sigma_1(\Eink)g-\Eink$ is positive. A simple direct computation shows that 
$$t_1(\Eink)=(n-k-1){\rm Ein}_{\frac{-k}{n-1-k}}.$$

\end{proof}
\begin{corollary}
For each $n-2 \leq k<n$ we have 
\begin{equation}
\Eink>0 \implies {\rm Ein}_{-k}>0.
\end{equation}
In particular, the positivity of $\Eink$ for $n-2 \leq k<n$ implies that the eigenvalues of the Ricci curvature are $k$-dispersed in the sense of Polombo \cite{Polombo}.
\end{corollary}
\begin{proof}
For $n-2\leq k< n-1$ the corrollary follows directly from the previous proposition as in this case $\frac{-k}{n-1-k}<-k$, For the remaining cases where $k\geq n-1$, the Ricci curvature must be positive and therefore all the $\Eink$ are positive for any negative $k$.
\end{proof}

\subsection{The positivity of $\Eink$ curvatures and the positivity index of Ricci curvature}

\begin{proposition}\label{number-pos-eigen}
Let $k$ be an integer such that $0\leq k\leq n-1$. The positivity of the $\Eink$ curvature implies the positivity of at least $k+1$ eigenvalues of the Ricci curvature.
\end{proposition}

\begin{proof}
Proposition \ref{k-positive-Ricci} implies that the Ricci curvature is $(n-k)$-positive. Denote by $\rho_1\leq \rho_2\leq ...\leq\rho_n$ the eigenvalues of the Ricci curvature, then
\[\rho_1+\rho_2+...+\rho_{n-k}>0.\]
Consequently, $\rho_{n-k}>0$ and therefore the eigenvalues $\rho_j>0$ for $j\geq n-k$.
\end{proof}
We remark that the previous proposition is optimal as the product $S^p\times H^q$ has positive $\Eink$ for $k=p-1$ and its Ricci curvature  has only $p$ positive eigenvalues.
\subsection{The positivity of the top $\Eink$ curvature}
By the top $\Eink$ we mean the tensor ${\rm Ein}_{n-1}=\Scal.g-(n-1)\Ric$ where $n$ is the dimension of the manifold. Recall that ${\rm Ein}_{n}$ cannot have a constant non-zero sign as it trace free.
The following proposition results directly  from propositions \ref{k-positive-Ricci} and \ref{hered} 
\begin{proposition}
For a Riemannian $n$-manifold, the positivity of the $\Eink$ curvature for $k=n-1$ implies at the same time the positivity of the Ricci curvature and the positivity of all the $\Eink$ curvatures for $k\in (-\infty, n-1]$.
\end{proposition}
Let us denote by $B$ the curvature ${\rm Ein}_{n-1}$, that is let $B={\rm Ein}_{n-1}$. We shall say that $\Gamma_2(B)>0$ if the first two elementary symmetric functions in the eigenvalues of $B$ are positive, that is if $\sigma_1(B)>0$ and $\sigma_2(B)>0$. A simple computation shows that
\begin{equation}
\sigma_1(B)=\Scal\,\,\, {\rm and}\,\,\, \sigma_2(B)=\frac{n-1}{2}\bigl(\Scal^2-(n-1)|\Ric|^2\bigr).
\end{equation}
In particular, 
\begin{equation}
\Gamma_2(B)>0 \iff \Scal>0 \,\,\, {\rm and}\,\,\, \Scal^2>(n-1)|\Ric|^2.
\end{equation}

We have the following refinement of the above proposition

\begin{proposition}
The Ricci curvature (more accurately $(n-1)\Ric$) coincides with the first newton transformation of ${\rm Ein}_{n-1}$, that is $(n-1)\Ric =t_1(B)$. Furthermore, if $\Gamma_2(B)>0$ then the Ricci curvature is positive and the modified Einstein tensors $\Eink$ are all positive for $k\in (-\infty, \frac{n}{2}]$. .
\end{proposition}
\begin{proof}
The first part is easy to check, in fact $t_1(B)=\sigma_1(B)g-B=(n-1)\Ric.$ It is classical that if $\Gamma_2(B)>0$ then the first Newton transformation $t_1(B)=(n-1)\Ric$ is positive definite. To prove the second part, we define the tensor 
$$\bar{B}:=\frac{2-n}{n}\Scal.g+(n-1)\Ric.$$
A straight forward computation shows that $\sigma_1(\bar{B})=\sigma_1(B)$ and $\sigma_2(\bar{B})=\sigma_2(B)$. Therefore, the positivity of $\Gamma_2(B)$ implies the positivity of $\Gamma_2(\bar{B})$ and consequently  the positivity of the first Newton transformation $t_1(\bar{B})$. A straightforward computation shows that
$$t_1(\bar{B})=\sigma_1(\bar{B})g-\bar{B}=\frac{2(n-1)}{n}\bigl(\Scal -\frac{n}{2}\Ric\bigr)=\frac{2(n-1)}{n}{\rm Ein}_{n/2}.$$
 This completes the proof.

\end{proof}
  
Recall that the Schouten tensor $A$ is defined by $A=\frac{1}{n-2}\bigl( \Ric-\frac{\Scal}{2(n-1)}.g\bigr)$. Here also we shall say that  $\Gamma_2(A)>0$ if $\sigma_1(A)>0$ and $\sigma_2(A)>0$. A simple computation shows that
\begin{equation}
\Gamma_2(A)>0 \iff \Scal>0 \,\,\, {\rm and}\,\,\, \frac{n}{4}\Scal^2>(n-1)|\Ric|^2.
\end{equation}
The proof of the following proposition is straightforward
\begin{proposition}
For $n\geq 4$ we have
$${\rm Ein}_{n-1}>0\implies \Gamma_2(B)>0\implies \Gamma_2(A)>0.$$
Furthermore, if $n=4$ one has
 $$ \Gamma_2(A)>0 \iff  \Gamma_2({\rm Ein}_3 )>0.$$
\end{proposition}
The following proposition generalizes the above result
\begin{proposition}
For $n\geq 4$ and $k=\frac{2n(n-1)}{3n-4}$ we have
$$ \Gamma_2(A)>0 \iff  \Gamma_2(\Eink)>0.$$

\end{proposition}

\begin{proof}
we define the tensor 
$$\bar{A}:=\frac{3n-4}{2n(n-1)(n-2)}\bigl(\Scal.g-\frac{2n(n-1)}{3n-4}\Ric\bigr).$$
A straight forward computation shows that $\sigma_1(\bar{A})=\sigma_1(A)$ and $\sigma_2(\bar{A})=\sigma_2(A)$. Therefore, the positivity of $\Gamma_2(A)$ is equivalent to the positivity of $\Gamma_2(\bar{A})$.
 This completes the proof.
\end{proof}
Note that in the previous proposition  $k=\frac{2n(n-1)}{3n-4}<n-1$ if $n\geq 5$.  If we increase $k$ we get the positivity of all higher $\Gamma_k(A)$ curvatures as shown by the next proposition.
\begin{proposition}
Let $A$ denotes the Schouten tensor. For $n\geq 4$ and $k=\frac{2(n-1)^2}{2n-3}$ we have
$$\Eink >0 \implies  A>0.$$
In particular, the positivity of $\Eink$ implies the positivity of all $\Gamma_i(A)$ curvature for $1\leq i\leq n$.
\end{proposition}
Note that in the previous proposition $k=\frac{2(n-1)^2}{2n-3}\in (n-1, n)$.
\begin{proof}
A straightforward computation shows that the Schouten tensor $A$ coincides, up to a positive constant factor, with the first Newton transformation of the tensor $\Eink$, where $k$ is as in the theorem. Precisely, one has
$$k(n-2)A=t_1(\Eink).$$
\end{proof}
\begin{remark}
The previous theorem is still true under the weaker condition $\Gamma_2(\Eink)>0$.
\end{remark}

\section{Proof of Theorems}
\subsection{Riemannian submersions and positive $\Eink$ curvature, proof of Theorem A }\label{Riem-subm-section}

The property $\Eink >0$  behaves very well with Riemannian products and as well with Riemannian submersions:

\begin{theorem}\label{Riem-subm}
Let $\pi  : M \to  B$  be a Riemannian submersion where the total
space $(M; g)$  is a compact manifold with $\dim M = n$  and the fibers  are with dimension $p$. 
Let  $\hat{g}$  denote the induced metric on the  fibers . Assume that  the fibers with the induced metric have positive $\Eink$ curvature for some $k<p$. Then the canonical variation metric $g_t$ of the total space has positive $\Eink$ curvature for all $t<t_0$ for some positive $t_0$.
\end{theorem}

\begin{proof}
We use the notations of \cite{Labbi-einstein} where the same property was proved for the case $k=2$. Let $E$ be a $g_t$-unit tangent vector to the total space, denote by $\alpha U$ its orthogonal projection onto the vertical subspace, where $U$ is a $g_t$-unit tangent vertical vector. Then $U_1=tU$ is $g_1$-unit tangent vertical vector and we have \cite{Labbi-einstein}
\begin{equation}
\begin{split}
(\Eink)_t(E)=&\Scal_t-k\Ric_t(E)=\frac{1}{t^2}\hat{\Scal}-\frac{k\alpha^2}{t^2}\hat{\Ric}(U_1)+{\rm O}(\frac{1}{t})\\
=& \frac{\alpha^2}{t^2}\bigl( \hat{\Scal}-k\hat{\Ric}(U_1)\bigr)+\frac{1-\alpha^2}{t^2}\hat{\Scal}+{\rm O}(\frac{1}{t}).
\end{split}
\end{equation}
Recall that for $k<p$, the positivity of $\hat{\Eink}$ implies the positivity of $\hat{\Scal}$ and note that $\alpha^2\leq 1$. The theorem follows using the compactness of $M$.
\end{proof}
\begin{corollary}
Let $S^p(\lambda)$ denotes the $p$-dimensional round sphere of radius $\lambda$ and $M$ be an arbitrary compact Riemannian manifold. Then the product 
$S^p(\lambda)\times M$ is with positive $\Eink$ for $p>1$, $k<p$ and $\lambda$ small enough. \\
In particular, the product $S^2(\lambda)\times M$ admits Riemannian metrics with  positive $\Eink$ for any $k\in (-\infty , 2)$.
\end{corollary}
As a direct consequence of the previous corollary, any finitely presented group can be made a the fundamental group of a compact Riemannian manifold with dimension $\geq 6$ and of  positive $\Eink$ for any $k\in (-\infty , 2)$. We can say better after the surgery theorem below.\\
\subsection{Proof of Theorem B}
As a second consequence of the previous theorem we have the following generalization of a result by Lawson-Yau \cite{Lawson-Yau}

\begin{theorem}
If a compact connected manifold $M$ admits  an effective action of a non-abelian  compact connected Lie group  $G$ then $\cEin(M)\geq 2$ and $\cein(M)=-\infty$.
\end{theorem}

\begin{proof}
The proof below follows almost word by word Lawson-Yau's proof of a similar result about positive scalar curvature \cite{Lawson-Yau}.\\
Since a non-abelian  compact connected Lie group contains a Lie subgroup isomorphic to $SU(2)$ or $ SO(3)$, 
we then can assume  without loss of generality that the group $G$ is $SU(2)$ or $ SO(3)$.\\
It is easy to see that $\cEin(G)\geq 2$ (in fact even $\geq 3$) and $\cein(G)=-\infty$.\\
If the action of $G$ is free, the proof is rather easy and goes as follows.
The canonical projection $M\to M/G$  is in this case a smooth submersion.
Let the fibers be equipped with a bi-invariant metric from the group $G$ via the
canonical inclusion ${\mathcal{G}} \subset  T_mM$.
Using any $G$-invariant metric on $ M$, we define the horizontal distribution to which
we lift up an arbitrary metric from the base $ M/G$. In this way we get  a Riemannian metric on
$M$ such that the projection $M\to M/G$ is a Riemannian submersion.\\
Finally, since  $\cEin(G)\geq 2$ and $\cein(G)=-\infty$  then  the fibers with the induced metric have as well the same property. We therefore conclude using the above Riemannian submersions Theorem \ref{Riem-subm}.\\
The general case, where the action is not free is more delicate and can be achieved in the following way. Consider the diagonal action of $G$ on $M\times G$, this action is free and its orbit space is $M$. Denote by $\pi:M\times G\to M$ the quotient map. We then fix a $G$-invariant metric $h$ on $M$ and we construct as above a metric $g$ on $M\times G$ such that $\pi:(M\times G,g)\to (M,h)$ is a Riemannian submersion. We denote by $g_t$ its canonical variation as in  thorem \ref{Riem-subm}. It turns out that the metrics $g_t$ are $G\times G$-invariant and thus make the projection onto the first factor $p_1:M\times G\to M$  a Riemannian submersion as well. Direct but long computations show that for $t$ small enough the submersed metric on $M$ has  positive $\Eink$ curvature for any $k<2$ and positive ${\rm ein}_k$ for any negative $k$. We bring the attention of the reader that a simplification of the difficult part of the computations near the fixed points was significantly  simplified in \cite{Labbi-groups}.
\end{proof}
\begin{remark}
In the above theorem, if one assumes the group to have a higher rank one can then concludes that the manifold has a higher $\cEin$ invariant, see a similar result about $p$-curvatures in \cite{Labbi-groups}.
\end{remark}
\subsection{A surgery theorem for positive $\Eink$ metrics, proof of Theorem C}\label{sec:surgery} 

First, we prove the following stability under surgeries theorem

\begin{theorem}\label{surgery-thm}
\begin{enumerate}
\item
For $k\in (-\infty, 1]$, the positivity of the $\Eink$ curvature is preserved under surgeries of codimensions $>2$.
\item For $k>1$, the positivity of the $\Eink$ curvature is preserved under surgeries of codimensions $>k+1$.
\end{enumerate}
\end{theorem}

\begin{proof} This theorem follows directly from Hoelzel's general surgery theorem  \cite{Hoelzel}. We  use the same notations as in \cite{Hoelzel}. Let $C_B({\Bbb{R}}^n)$ denote the  vector space of algebraic curvature operators $\Lambda^2 {\Bbb{R}}^n\rightarrow \Lambda^2 {\Bbb{R}}^n$ satisfying the first Bianchi identity and endowed with the canonical inner product.
Let 
$$C_{\Eink >0}:=\{R\in C_B({\Bbb{R}}^n): \Eink(R)>0\},$$
where for a unit vector $u$, $\Eink(R)(u)=\Scal(R)-k\Ric(u)$. Here $\Ric$ and  $\Scal(R)$ denote respectively  as usual the first Ricci contraction  and the full contraction of $R$. The subset $C_{\Eink >0}$ is clearly open, convex and it is an ${\rm{O}}(n)$-invariant cone. Furthermore, it is easy to check that $\Eink(S^{q-1}\times \Bbb{R}^{n-q+1})>0$ if $q>2$ and  $q>k+1$. This completes the proof.

\end{proof}

\subsection{The $\cEin$ and $\cein$ invariants of simply connected manifolds, proof of Theorem D}
According to whether the simply connected manifold is either spin or non-spin we have the following classification theorem
\begin{theorem}\label{simply-connected-thm}
Let $n\geq 5$ and $k\in(-\infty, 2)$. Then
\begin{enumerate}
\item Every compact simply connected and non-spin $n$-manifold admits a Riemannian metric with $\Eink>0$.
\item Every compact simply connected spin $n$-manifold admits a metric with $\Eink>0$ if and only if it admits a metric with positive scalar curvature.
\end{enumerate}
\end{theorem}

\begin{proof}
The theorem is a consequence of the proof of the Gromov-Lawson-Stolz classification theorem of compact simply connected manifolds with positive scalar curvature. In fact, from one side the surgery theorem guarantees the stability under surgeries in codimensions $\geq 3$ of $\Eink>0$ for $k\in(-\infty, 2)$. From the other side, Theorem \ref{Riem-subm} guarantees that the generators of the oriented bordism group as in \cite{GroLa} and the ${\Bbb H}{\rm P}^2$ bundles used in \cite{Stolz} admit metrics with positive 
$\Eink>0$ for $k\in(-\infty, 2)$.
\end{proof}
\subsection{The $\cEin$ invariant of $2$-connected manifolds, proof of Theorem E part 1}
In this case the positivity of the scalar curvature can be upgraded to positive $\Eink$ curvature for any $k\in(-\infty, 3)$ as follows
\begin{theorem}
Let $k\in(-\infty, 3)$. A compact $2$-connected manifold of dimension $n\geq 7$ admits a metric with positive $\Eink$ curvature if and only if it admits a metric with positive scalar curvature.
\end{theorem}
\begin{proof}
Positive  $\Eink$ curvature always implies positive scalar curvature. To prove the converse, first recall that a $2$-connected manifold has a canonical spin structure. We then use a result of Stolz \cite{Stolz} which asserts that a spin manifold with positive scalar curvature is spin cobordant  to the total space of a bundle with the quaternionic projective plane ${\Bbb H}{\rm P}^2$ as fibre and structural group the group of isometries of ${\Bbb H}{\rm P}^2$. These total spaces have positive $\Eink$ curvatures for $k\in(-\infty, 3)$ by Theorem \ref{Riem-subm}. On the other hand, Lemma 4.2 in \cite{agag} asserts that if $M_1$ is a compact 2-connected manifold of dimension $n\geq 7$ that is spin cobordant to a manifold $M_0$, then the manifold $M_1$ can be obtained from $M_0$ by surgeries of codimension $\geq 4$. To complete the proof recall that the surgery theorem \ref{surgery-thm} guarantees the stability under surgeries in codimensions $\geq 4$ of the $\Eink$ curvatures for $k\in(-\infty, 3)$.

\end{proof}
\begin{remark}
It is known that a compact $2$-connected $6$-manifold is diffeomorphic to a connected sum of copies of $S^3\times S^3$ and therefore always admits Riemannian metrics with positive $\Eink$ curvature for $k\in(-\infty, 5)$ by Theorem \ref{surgery-thm}.

\end{remark}
\subsection{The $\cEin$ invariant of $3$-connected manifolds, proof of Theorem E part 2}

Let $p_1(M)$ denotes the first Pontryagin class of the manifold $M$. Recall that a  spin manifold $M$ is said to be  a string manifold if $\frac{1}{2}p_1(M)=0$.Otherwise,
we say that $M$ is a non-string manifold. In the later case, one can upgrade positive scalar curvature to positive $\Eink$ curvature for any $k\in(-\infty, 4)$ as follows
\begin{theorem}
Let $k\in(-\infty, 4)$. A compact $3$-connected non-string manifold of dimension $n\geq 9$ admits a metric with positive $\Eink$ curvature if and only if it admits a metric with positive scalar curvature. In particular, a compact $3$-connected non-string manifold of dimension $n\geq 9$ that accepts a metric with  positive scalar curvature admits a metric with positive ${\rm Ein}_3$ curvature.
\end{theorem}
\begin{proof}
We proceed as in the above proof. 
First note that positive  $\Eink$ curvature always implies positive scalar curvature. To prove the converse, we recall that a $3$-connected manifold has a canonical spin structure. We then use a result of Stolz \cite{Stolz} which asserts that a spin manifold with positive scalar curvature is spin cobordant  to the total space of a bundle with the quaternionic projective plane ${\Bbb H}{\rm P}^2$ as fibre and structural group the group of isometries of ${\Bbb H}{\rm P}^2$. These total spaces have positive $\Eink$ curvatures for $k\in(-\infty, 4)$ by Theorem \ref{Riem-subm}. On the other hand, Proposition 3.7  in \cite{Bot-Lab1} asserts that if $M_1$ is a compact $3$-connected  and non-string manifold of dimension $n\geq 9$ that is spin cobordant to a manifold $M_0$, then the manifold $M_1$ can be obtained from $M_0$ by surgeries of codimension $\geq 5$. To complete the proof recall that the surgery theorem \ref{surgery-thm} guarantees the stability under surgeries in codimensions $\geq 5$ of the $\Eink$ curvatures for $k\in(-\infty, 4)$.
\end{proof}
Next we discuss the case of $3$-connected and string manifolds. We denote by $\Omega_*^\str$ the string cobordism ring. We denote by  $I_3$ the subset of  $\Omega_*^\str$ 
consisting of of string cobordism classes that contain a
representative with positive   ${\rm Ein}_3$ curvature. Recall that  the
cartesian product of a manifold of positive ${\rm Ein}_3$  curvature with an arbitrary manifold
has always a metric with  positive ${\rm Ein}_3$ curvature, see section \ref{Riem-subm-section}. Therefore, the subset $I_3$ is an ideal of the ring $\Omega_*^\str$.
Denote by $\alpha_3$ the quotient map which is a ring homomorphism
\begin{equation}
\alpha_3: \Omega_*^\str\longrightarrow \Omega_*^\str/I_3.
\end{equation}
\begin{proposition}
A compact $3$-connected string manifold $M$ of dimension $n\geq 9$ admits a metric with positive ${\rm Ein}_3$ curvature if and only if $\alpha_3([M])=0$.
\end{proposition}
\begin{proof}
If $M$ has positive ${\rm Ein}_3$ curvature then its string cobordism class $[M]$ belongs to the ideal $I_3$ and therefore $\alpha_3([M])=0$ by definition. Conversely, suppose $\alpha_3([M])=0$, then $[M]\in I_3$ and therefore the manifold $M$ is string cobordant to a manifold $M_0$ with positive ${\rm Ein}_3$ curvature. On the other hand, Proposition 3.4 in \cite{Bot-Lab1} asserts that $M$ can be obtained from $M_0$ by surgeries of codimension $\geq 5$. We conclude using the surgery theorem \ref{surgery-thm} that $M$ has  positive ${\rm Ein}_3$ curvature. 
\end{proof}
We remark that one can continue this process of requiring higher connectivity from the manifold and getting stronger curvature positivity, see \cite{Bot-Lab1}.
At the end of this section we ask the following natural questions: 
\begin{itemize}
\item Does there exist any compact 3-connected string  Riemannian manifold $M$ of dimension $n\geq 9$ with positive scalar curvature such that $\alpha_3[M]\not=0$? That is (equivalently) such that $M$ does not admit any metric with positive ${\rm Ein}_3$ curvature.
\item Is there any topological description of the genus $\alpha_3$? What's its relation (if any) with the Witten genus?
\end{itemize}

\subsection{Fundamental group of manifolds with  positive 
$\Eink$-curvature, proof of Theorem F}\label{fund:group}

\begin{theorem}\label{fund-group} Let $\pi$ be a finitely presented group. Then
 for every $n\geq 4$ and for every $k\in(-\infty, n-2)$, there exists a compact $ n$-manifold $M$ with
positive $\Eink$-curvature such that $\pi_1(M)=\pi$.

\end{theorem}

\begin{proof}
Let $n\geq 4$ and $\pi$ be a group which has a presentation consisting
of $k$ generators $x_{1},x_{2},...,x_{k}$ and $\ell$ relations
$r_{1},r_{2},...,r_{\ell}$.  Let the manifold $S^{1}\times S^{n-1}$ be
given a standard product metric which has positive
$\Eink$-curvature.  Since $\pi_1(S^{1}\times S^{n-1})\cong \Z$, the
Van-Kampen theorem implies that the fundamental group of the connected
sum
$$
N:= \#k (S^{1}\times S^{n-1})
$$ is a free group on $k$ generators, which we denote by
$x_{1},x_{2},...,x_{k}$. By the above surgery Theorem, $N$
admits a metric with positive $\Eink$-curvature because in both cases $n>k+1$.

We now perform surgery $\ell$-times on the manifold $N$ such that each
surgery is of codimension $n-1> k+1$, killing in succession the
elements $r_{1},r_{2},...,r_{\ell}$. Again, according to the above surgery theorem, the resulting manifold $M$ has fundamental group
$\pi_1(M)\cong \pi$ and admits a metric with positive
$\Eink$-curvature, as desired.

\end{proof}
\begin{remark}
For the top $\Eink$ where $k\in [n-1, n)$, the positivity of $\Eink$ implies the positivity of the Ricci curvature and therefore forces the fundamental group to be finite. We conjecture that the fundamental group of a compact manifold with positive $\Eink$ for $k\in [n-2, n-1)$ must be virtually free, see the last section.
\end{remark} 
\section{$\cEin$ and $\cein$ invariants in low dimensions}
\subsection{The case of two and three dimensions}
\begin{proposition}
\begin{itemize}
\item For a compact manifold $M$ of dimension $2$ we have  $\cEin(M)=0 \,\, {\rm or}\,\, 2$ and $\cein(M)=0 \,\, {\rm or}\, -\infty$.
\item For a compact manifold $M$ of dimension $3$ we have  $\cEin(M)=0 , \, 2\,\, {\rm or}\, \,3$ and $\cein(M)=0 \,\, {\rm or}\, -\infty$.
\end{itemize}
\end{proposition}
\begin{proof}
Let $M$ be a compact manifold of dimension $2$. If $\cEin(M)>0$  then $M$ has a metric with positive scalar curvature and consequently has a metric with constant curvature by the uniformization theorem and therefore $\cEin(M)=2$.\\
Similarly, if a compact manifold $M$ of dimension $3$ has $\cEin(M)>0$  then it has a metric with positive scalar curvature and consequently it is diffeomorphic to  connected sums of spherical  manifolds and copies of $S^2\times S^1$. The $\cEin$ invariant of a spherical manifold is $3$ and $\cEin(S^2\times S^1)=2$. The connected sum theorem that we proved in this paper shows then that  $\cEin(M)\geq 2$. If $\cEin(M)> 2$ then it admits a metric with positive Ricci curvature and therefore by Hamilton's theorem it has a constant curvature metric so that $\cEin(M)=3$. The proof for $\cein(M)$ is completely similar.
\end{proof}

\subsection{$\cEin$ and $\cein$ invariants of four dimensional manifolds}

 Suppose that $M$ is a compact oriented $4$-manifold. Recall that the Gauss-Bonnet formula asserts that the Euler-Poincar\'e characteristic is given by an integral of a quadratic curvature invariant. It can be stated in the following form
\begin{equation}
\chi(M)=\frac{1}{8\pi^2}\int_M\biggl(||W||^2+\frac{1}{9}\sigma_2({\rm Ein}_3(g))\biggr)\mu_g.
\end{equation}
Where $||W||$ is the norm of the Weyl curvature tensor and $\sigma_2({\rm Ein}_3(g))= \frac{3}{2}\Scal^2-\frac{9}{2}||\Ric||^2$ is the second symmetric function in the eigenvalues of the tensor ${\rm Ein}_3(g)$ . We remark that  $\sigma_2({\rm Ein}_3(g))=36\sigma_2(A(g))$, where $\sigma_2(A(g))$ is the second symmetric function in the eigenvalues of the Schouten tensor. We immediately conclude that
\begin{equation}
\cEin(M)>3 \implies \chi(M)>0.
\end{equation}
In particular, $S^2\times T^2$ and $S^3\times S^1$ do not admit metrics with positive $\Eink$ for $k\geq 3$. These examples don't admit positive Ricci curvature metrics too. Below we will show that we may have positive Ricci curvature but no metric with ${\rm Ein}_3>0$.\\
Recall that the first Pontryagin number $p_1(M)$ of $M$ satisfies (see for instance remark 4.13 in \cite{Labbi-JAMS}).

\begin{equation}
|p_1(M)|\leq \frac{1}{4\pi^2}\int_M ||W||^2\mu_g.
\end{equation}
We immediately conclude that
\begin{equation}
2\chi(M)-|p_1(M)|\geq \frac{1}{18\pi^2}\int_M \sigma_2({\rm Ein}_3(g))\mu_g.
\end{equation}
Since the signature of $M$ satisfies $3\sigma(M)=p_1(M)$, we have therefore proved the following

\begin{proposition}
If a compact oriented $4$-manifold admits a metric with positive  ${\rm Ein}_3$ curvature  then its Euler-Poincar\'e characteristic  and signature satisfy the inequality
\begin{equation}
2\chi(M)-3|\sigma(M)| > 0.
\end{equation}
\end{proposition}

As an application, we consider the following family of 4-manifolds
 $$M_{kl}=k({\Bbb C}P^2)\# l\bigl({\overline{{\Bbb C}P^2}}\bigr)$$
 that are obtained by taking connected sums of $k$-copies of the $2$-complex projective space ${\Bbb C}P^2$ together with $l$ copies of the $2$-complex projective space with reversed orientation. 
\begin{proposition}\label{cp2-connected-sum}
\begin{itemize}
\item For $k>5l+3$ or $l>5k+3$ we have $\cEin(M_{kl})=3$.
\item For $k=1$ and $0\leq l\leq 8$ we have  $\cEin(M_{kl})=4$.
\item For $k>3$ we have   $\Ein\left( k({\Bbb C}P^2)\right)=3$.
\end{itemize}
\end{proposition}
\begin{proof}
First note that the surgery theorem implies that $\cEin(M_{kl})\geq 3$. The Euler-Poincar\'e characterestic and the signature of $M_{kl}$ are respectively given by
$$\chi(M_{kl})=k+l+2\,\, {\rm and}\,\, \tau(M)=k-l.$$
If $\cEin(M_{kl})>3$ then there exists a metric $g$ on of $M_{kl}$ with ${\rm Ein}_3(g)>0$,  consequently, the above proposition shows that
$$2k+2l+4 > 3|k-l|.$$
If $k>l$, the previous inequality reads $k<5l+4$ or $k\leq 5l+3$. Similarly,
if $l>k$, the previous inequality reads $l<5k+4$ or $l\leq 5k+3$. This completes the proof of the first part.The proof of the third part is similar. The manifolds of the second part all admit Einstein metrics with positive scalar curvature and therefore their $\cEin$ invariant is 4. Precisely, all the following Riemannian manifolds are Einstein with positive scalar curvature: ${\Bbb C}P^2$ with the Fubiny-Study metric, ${\Bbb C}P^2\# {\overline{{\Bbb C}P^2}}$ with the Page metric \cite{Page}, ${\Bbb C}P^2\# 2\bigl({\overline{{\Bbb C}P^2}}\bigr)$ with the Chen-Lebrun-Weber metric \cite{CLW} and finally ${\Bbb C}P^2\# l\bigl({\overline{{\Bbb C}P^2}}\bigr)$ for $3\leq l\leq 8$, with Tian-Yau metric \cite{Tian-Yau}.
\end{proof}

\section{ Final remarks and open questions}
\subsection{Extremal manifolds $M$ with $\cEin(M)=\dim M$ or $\cein(M)=-\infty$}
Recall that for an Einstein $n$-manifold $(M,g)$ with positive scalar curvature one has ${\cEin}(M)=n$. Conversely, if for some compact manifold one has ${\rm Ein}(M)=n$, then one has a sequence $g_k$ of Riemannian metrics on $M$ such that $\lim_{k\to\infty}\Ein(g_k)=n$. We remark that this sequence consists of metrics with positive Ricci curvature as we can suppose that  $\Ein(g_k)>n-1$ for all $k$. If this sequence of Riemannian metrics  (or a subsequence of it) converges to a smooth  metric $h$ on $M$ then this metric must be Einstein, in fact one should have $\Ein(h)=n$, this is possible only if $h$ is an  Einstein metric. It is an open question to decide whether there are compact $n$-manifolds $M$ with ${\rm Ein}(M)=\dim(M)$ but do not admit any Einstein metric with positive scalar curvature.\\
On the other hand, recall that if a manifold $M$ has non-negative Ricci curvature and positive scalar curvature metric then $\cein(M)=-\infty$. One may ask as above about the converse. The answer here is no as shown by the following example. Consider the manifold $N$ of dimension $n\geq 3$ consisting of several connected sums of $S^{n-1}\times S^1$ such that it first betti number satisfies $b_1(N)>n$. It is clear that $\cein(M)=-\infty$ but $N$ does not admit any metric with non-negative Ricci curvature by Bochner's theorem. The point here, in contrast with the above case, is that the condition $\ein(g_k)$ negative and high for a given sequence of metrics is not strong enough to guarantee any possible convergence of the sequence.
\subsection{Manifolds with intermediate $\cEin$ invariants}
Recall that for a compact manifold $M$ which doesn't have a metric with positive scalar curvature satisfies $\cEin(M)=0$, and if the the manifold $M$ has a positive scalar curvature Einstein metric then $\cEin(M)=n$. We now seek for $n$-manifolds $M$ for which $0<\cEin(M)<n$.\\
Remark that for $n\geq 2$ one has 
\begin{equation}
{\rm Ein}(S^n\times S^1)=n.
\end{equation}
Where $S^n$ is the standard sphere of dimension $n$.  In fact, the product metric has $\Eink>0$ for all $k<n$ so that ${\rm Ein}(S^n\times S^1)\geq n.$ On the other side the product can't carry metrics of positive Ricci curvature and therefore can't admit any metric with ${\rm Ein}_n>0$, consequently, ${\rm Ein}(S^n\times S^1)\leq n.$ \\
We conjecture that more generally for any positive $k$ and for $n\geq 2$, one has
\begin{equation}\label{SnTk}
{\rm Ein}(S^n\times T^k)=n.
\end{equation}
Where $T^k=S^1\times S^1\times ...\times S^1$ is the torus of dimension $k$. This would provide intermediate integer values for $\cEin$ from $2$ to $n-1$. For the remaining values we ask the following questions

\begin{itemize}
\item Is there any compact manifold $M$  of dimension $\geq 2$ with  $0<\cEin(M)<2$? 
\item  Is there any compact manifold $M$ with ${\rm Ein}(M)$ not an integer?
\end{itemize}
In dimensions $2$ and $3$ such manifolds do not exist as we have seen in this paper. However, if we redefine $\cEin$ with respect to a fixed conformal class of a given metric we can get such examples as we will explain in a forthcoming separate paper \cite{conformal-Eink}.

\subsection{Fundamental group and $\cEin$ invariant}
 Theorem \ref{fund-group} shows that there are no restrictions on the fundamental group of a compact $n$-manifold with ${\rm Ein}(M)\leq n-2$. The condition ${\rm Ein}(M)>n-1$ implies that the fundamental group is finite. However, the condition $n-2<{\rm Ein}(M)\leq n-1$, implies that the Ricci curvature is 2-positive by proposition
\ref{k-positive-Ricci}. It is an open question of Wolfson \cite{Wolfson} that this later condition forces the fundamental group to be virtually free. We then ask the following weaker question\\
Is the fundamental group of a compact $n$-manifold $M$ with $n-2<{\rm Ein}(M)\leq n-1$  necessarily virtually free?\\
If the above statement is true, then it will provides a direct proof of the statement  in conjecture \ref{SnTk}.
\subsection{$\cEin$ invariant and macroscopic codimension}

For a metric space $X$, following Gromov, we say that
$\dim_\epsilon X \leq k$ if there is a $k-$dimensional polyhedron $P$ and
a continuous map $f: X\rightarrow P$ such that the $\Diam(f^{-1}(p))
\leq \epsilon$ for all $p\in P$. A metric space $X$ has macroscopic
dimension $\dim_{mc}X\leq k$ if $\dim_\epsilon\leq k$ for some
possibly large $\epsilon <\infty$.  If $k$ is minimal, we say that
$\dim_{mc}X= k.$\\
 For instance, $\dim_{mc}S^n\times {\Bbb R}^k=k$ is equal to $k$, see \S 2 in \cite{Gromov2}.\\
On the other hand,  Proposition \ref{number-pos-eigen}, shows that if ${\cEin}(M)>k$, then the $n$-manifold $M$ has a Riemannian metric with $k+1$ positive eigenvalues. In a recent paper \cite{Wolfson2},  Wolfson proved (there is a gap in the proof)  that the macroscopic dimension (of its universal cover) of such an $n$-manifold should be $\leq n-1-k$.\\
In view of the above two facts we are led to the following conjecture
\begin{equation}\label{mc-conj}
{\cEin}(M)\leq n-\dim_{mc}(\tilde{M}).
\end{equation}
Where $M$ is a compact manifold of dimension $n$ and $\tilde{M}$ is its universal cover. The inequality is trivial in case the manifold $M$ is simply connected. We note that  Bolotov \cite{Bolotov} constructed a non-simply connected $4$-manifold $M$ with macroscopic dimension of its universal cover equal to $3$ but does not carry metrics with positive scalar curvature, that is with ${\rm Ein}(M)=0$. \\

If conjecture \ref{mc-conj} turns out to be true together with the conjecture  $\cEin(M)>0\implies {\cEin}(M)\geq 2$, these would imply a positive answer to Gromov's conjecture that $\dim_{{mc}}(\tilde{M})\leq n-2$ for a manifold $M$ with positive scalar curvature.\\

\subsection{$\Ein$ invariant and Berger number}
Berger \cite{Berger} defined the number ${\rm min}\int |R|^2\mu_g$ for a compact manifold, it is defined as the infinimum over all possible Riemannian metrics on a compact manifold of the square of the $L^2$ norm of the Riemann tensor. We remark that for a four dimensional compact manifold $M$ one has
\[\cEin(M)=4 \implies {\rm min}\int |R|^2\mu_g=8\pi^2\chi(M).\]
 To prove this fact we use the following form of the Gauss-Bonnet formula
\[\int|R|^2\mu_g=8\pi^2\chi(M)+\int\left(|\Ric|^2-\frac{\Scal^2}{4}\right)\mu_g.\]
The last term is nonnegative as it is the square of the norm of the trace free Ricci tensor and so $\int|R|^2\mu_g\geq 8\pi^2\chi(M)$ for any metric on $M$. On the other hand if $\cEin(M)=4$ then one has from the definition a sequence of Riemannian metrics on $M$ which renders the above last term as small as we wish.

\subsection{Extremal vs. best Riemannian metrics on a compact manifold}
On a given compact manifold $M$ with $\cEin(M)>0$, a Riemannian metric $g$ such that $\Ein(g)=\cEin(M)$ shall be called an {\sl extremal Riemannian metric} on $M$. For instance Einstein metrics with positive scalar curvature on any compact manifold are all extremal Riemannian metrics. The standard product metric on $S^n\times S^1$ for $n\geq 2$ is an extremal metric on $S^n\times S^1$ that is not Einstein. We conjecture that the standard product metrics on $S^n\times M^q$ are extremal metrics for $n\geq 2$ and $M^q$ is a space form with curvature $\leq 0$.\\
It would be interesting to establish a variational characterization of these extremal metrics and to study their existence on any given compact manifold. Similar questions may be asked for metrics such that $\cein(M)=\cein(g)$.
\appendix
\section{Modified Schouten tensors}

In a similar way, we define the modified Schouten tensors to be
\begin{equation}
\Schoutk=\Ric-k\Scal.g.
\end{equation}
We recover the Schouten tensor for $k=\frac{1}{2(n-1)}$ where $n$ is the dimension of the underlying manifold. The trace of $\Schoutk$ is $(1-nk)\Scal$, the positivity of the tensor $\Schoutk$ implies the positivity of the scalar curvature as far as $k<\frac{1}{n}$. It implies as well the positivity of the Ricci tensor if $0<k<1/n$. Below we will show that the positivity of $\Schoutk$ is equivalent to the positivity of some ${\rm Ein}_r$.\\
\begin{proposition}
For $k<0$ we have
\begin{equation}
\Schoutk>0\iff {\rm Ein}_{-\frac{1}{k}}>0.
\end{equation}
\end{proposition}
\begin{proof}
Straightforward after remarking that $\Schoutk=-k(\Scal\, g-\frac{1}{k}\Ric).$
\end{proof}
\begin{proposition}
For $n-1\leq k<n$ where $n$ is the dimension of the underlying manifold we have
\begin{equation}
\Eink>0\implies {\rm Sch}_{1-\frac{n-1}{k}}>0.
\end{equation}
\end{proposition}
Note that as $k$ ranges from $n-1$ to $n$, the index $1-\frac{n-1}{k}$ ranges from $0$ to $\frac{1}{n}$.
\begin{proof}
We compute the first Newton transformation of $\Eink$ as follows
\begin{equation}
t_1(\Eink)=\sigma_1(\Eink)\, g-\Eink=k\bigl(\Ric-(1-\frac{n-1}{k})\Scal\, g\bigr).
\end{equation}
Actually one needs only $\Gamma_2(\Eink)$ to be positive in order to guarantee the positivity of ${\mathrm Sch}_{1-\frac{n-1}{k}}$.
\end{proof}
Conversely, we have

\begin{proposition}
For $1\leq k<n$ where $n$ is the dimension of the underlying manifold we have
\begin{equation}
{\rm Sch}_{\frac{k-1}{k(n-1)}}>0\implies \Eink>0.
\end{equation}
\end{proposition}
Note that as $k$ ranges from $1$ to $n$, the index $\frac{k-1}{k(n-1)}$ ranges from $0$ to $\frac{1}{n}$.
\begin{proof}
Like in the above proof, we compute the first Newton transformation  as follows
\begin{equation}
\begin{split}
t_1({\rm Sch}_{\frac{k-1}{k(n-1)}})=&\sigma_1({\rm Sch}_{\frac{k-1}{k(n-1)}})\, g-{\rm Sch}_{\frac{k-1}{k(n-1)}}\\
=& \frac{1}{k}\Scal\, g-\Ric=\frac{1}{k}\Eink.
\end{split}
\end{equation}
Actually one needs only $\Gamma_2({\rm Sch}_{\frac{k-1}{k(n-1)}})$ to be positive in order to guarantee the positivity of $\Eink$.
\end{proof}

\end{document}